\documentclass[10pt,a4paper,oneside]{article}
\usepackage[letterpaper,hmargin=1.00in,vmargin=1.00in]{geometry}
\usepackage[utf8]{inputenc}
\usepackage[T1]{fontenc}
\usepackage{amsmath}
\usepackage{amsfonts}
\usepackage{amssymb}
\usepackage{array}
\usepackage{graphicx}
\usepackage{enumitem}
\usepackage{float}
\usepackage{graphicx}
\usepackage{amsthm}
\usepackage{xcolor}
\usepackage{csquotes}
\usepackage{tikz}
 \usepackage{mwe} 
\usepackage{caption}
\usepackage{cite}
\usepackage{hyperref}
\usepackage{authblk}
\usepackage{csquotes}
\usepackage{tikz}
\usepackage{subcaption}
\usepackage{mwe}
\usepackage{tabularray}
\UseTblrLibrary{amsmath}
\usetikzlibrary{shapes.geometric, arrows}

\newtheorem{theorem}{Theorem}
\newtheorem{remark}[theorem]{Remark}
\newtheorem{lemma}[theorem]{Lemma}

\newtheorem{proposition}[theorem]{Proposition}
\newtheorem{definition}[theorem]{Definition}

\begin{document}

\title{On Spectral Properties of Lanzhou Matrix of Graphs}
\author{Madhumitha K V, Harshitha A, Swati Nayak, Sabitha D'Souza \footnote{Corresponding author}}
\date{}

\maketitle

\begin{abstract}
Let $\Gamma$ be a simple graph on $n$ vertices. Lanzhou index is defined as $Lz(\Gamma)=\sum\limits_{u \in V(\Gamma)}d_\Gamma(u)^2d_{\overline{\Gamma}}(u).$ In this manuscript, the Lanzhou matrix, denoted by $A_{Lz}(\Gamma)$, has been defined, and its spectral properties are studied. The $uv^{th}$ entry in $A_{Lz}(\Gamma)$ is $d_\Gamma(u)d_{\overline{\Gamma}}(u)+d_\Gamma(v)d_{\overline{\Gamma}}(v)$ if $u$ and $v$ are adjacent. Otherwise, the entry is zero. Some bounds on Lanzhou energy and spread on the Lanzhou matrix are obtained. Also, Lanzhou eigenvalues and inertia for some standard graphs have been obtained. Additionally, characterizations for the symmetricity of Lanzhou eigenvalues about the origin are obtained.
\end{abstract}

\textbf{Keywords-} Lanzhou matrix, spread, Lanzhou energy, inertia

\section{Introduction}\label{1.1}
Spectral graph theory is a field of mathematics that connects graph theory and linear algebra. The eigenvalues and eigenvectors of matrices related to the graph, such as the adjacency matrix, Laplacian matrix, and their variations, are analysed in order to study the properties of graphs. The fundamental idea is that the structure of a graph can be understood by studying the spectrum (the set of eigenvalues) of these matrices.
	
The adjacency matrix of a graph encodes the connections between vertices, while the Laplacian matrix captures information about the graph’s connectivity and can be used to analyze important graph properties, such as clustering, connectivity, and graph partitioning. The eigenvalues and eigenvectors of these matrices give insights into key structural features of the graph, such as the number of connected components, the presence of clusters, and the graph’s robustness to disconnection.
	
The focus of the work is on undirected, simple graphs. A graph $\Gamma$ with vertex set $V$ and edge set $E$ is considered. Let $n$ and $m$ represent the order and size of $\Gamma,$ respectively.  $d_\Gamma(u)$ is a representation of the degree of a vertex $u$ in $\Gamma.$ Whenever the graph under consideration is clear from the context, $d_\Gamma(u)$ is replaced by $d(u).$  $\overline{\Gamma},$ represents the graph with the same vertex set $V,$ which is the complement of $\Gamma.$ It has two vertices that are adjacent if and only if they are not adjacent in $\Gamma.$ Matrix representations of graphs are convenient. The adjacency matrix $A(\Gamma),$ which is defined as $A(\Gamma) = [a_{ij}],$ is a commonly used matrix representation. If vertices $v_i$ and $v_j$ are adjacent, the $(i,j)^{th}$ entry is $1$; otherwise, it is $0.$ $\Gamma$'s eigenvalues are obtained from $A(\Gamma).$  $\rho(\Gamma)$ is the spectral radius, which is the largest eigenvalue. The spectrum of the graph $\Gamma$ is comprised of the eigenvalues with their multiplicities. The inertia of the symmetric matrix $A(\Gamma),$ is indicated by the triplet $(n_+,n_-,n_0)$, which is expressed as $In(A(\Gamma)),$ where $n_+,\ n_-,$ and $n_0$ are the number of positive, negative, and zero eigenvalues of $A(\Gamma),$ respectively. The graph energy of $\Gamma,$ initially coined by Gutman, I (see \cite{555}), is the sum of the absolute values of the eigenvalues of $A(\Gamma).$ According to HMO theory, the total $\pi$-electron energy of hydrocarbon molecules coincides with the energy of their chemical graphs. This connection between graph energy and chemistry has motivated researchers worldwide to define new matrices and explore their associated graph energies. One can refer \cite{8,9,10,11,12,13,14,15,16} to know more about different graph energies.

\begin{definition}
The first and the second Zagreb indices \cite{20} are given by
$$M_1(\Gamma)=\sum\limits_{u \in V(\Gamma)}d(u)^2=\sum\limits_{uv\in E(\Gamma)}d(u)+d(v).$$
and	
$$M_2(\Gamma)=\sum\limits_{uv\in E(\Gamma)}d(u)d(v).$$
\end{definition}
\begin{definition}
In \cite{21}, Furtula, B and Gutman, I proposed the Forgotten index as
$$F(\Gamma)=\sum\limits_{u\in V(\Gamma)}d(u)^3=\sum\limits_{uv\in E(\Gamma)}d(u)^2+d(v)^2.$$	
\end{definition}

Article \cite{1} introduces a new topological index called the Lanzhou index and investigates its properties. The Lanzhou index outperforms several existing indices in predicting a chemically relevant property, the octanol-water partition coefficient, on some benchmark datasets. The extremal values and extremal graphs, trees, and restricted trees (chemical trees) for the Lanzhou index are also determined and characterized.

It is defined as,
\begin{align*}
Lz(\Gamma)&=(n-1)M_1(\Gamma)-F(\Gamma)\\&=\sum\limits_{u \in V(\Gamma)}d_\Gamma(u)^2d_{\overline{\Gamma}}(u).
\end{align*}
Here, $d_{\overline{\Gamma}}(u)$ denotes the degree of the vertex $u$ in $\overline{\Gamma}.$ When the graph under consideration is clear from the context, $d_{\overline{\Gamma}}(u)$ is written as $d(\overline{u}).$ Readers are directed to \cite{2,3,4,5,6,7} for the in-depth descriptions of these graph indices.

\subsection{Motivation}
Many matrices in graph theory and graph-based machine learning are directly dependent on the adjacency relationships among the vertices. Examples include the adjacency matrix, which explicitly records which pairs of vertices are connected by edges. These matrices are central to many classical and modern graph algorithms, including spectral clustering, graph partitioning, and random walks. Most of the spectral properties of the weighted matrices directly follow from the adjacency spectral properties.

Here, we have a novel matrix, namely, the Lanzhou matrix which is not completely dependent on adjacency between the vertices in the graph. The independence of Lanzhou matrix from adjacency offers a critical advantage in graph-based learning models. Lanzhou matrix exhibit many spectral property that is significantly different from that of the adjacency matrix. This divergence makes Lanzhou matrix particularly interesting to study, as it captures dimensions of information that are orthogonal or complementary to the graph’s topology. This allows for a more modular design in graph algorithms, where structural and feature information can be treated and analyzed separately or in combination.

The Lanzhou matrix of graph $\Gamma,$ given by
$$A_{Lz}(\Gamma) =
\begin{cases} 
d_\Gamma(u)d_{\overline{\Gamma}}(u)+d_\Gamma(v)d_{\overline{\Gamma}}(v), &\text{if $u\sim v$ in $\Gamma$ } \\
	0, &\text{otherwise.}\\
\end{cases}$$
Let $\lambda_1(\Gamma)\geq \lambda_2(\Gamma) \geq \cdots \geq \lambda_n(\Gamma)$ denote the eigenvalues of $A_{Lz}(\Gamma).$ The set of all eigenvalues with its multiplicities are the Lanzhou spectrum of $\Gamma,$  denoted by $Lz-$ spectrum. The Lanzhou spectral radius $\rho_{Lz}(\Gamma)$ is the largest eigenvalue of $A_{Lz}(\Gamma).$ Since $A_{Lz}(\Gamma)$ is entry-wise non-negative, according to the Perron-Frobenius theorem, the spectral radius of $A_{Lz}(\Gamma),$ represented by $\rho(A_{Lz}(\Gamma)),$ is the Lanzhou spectral radius of $\Gamma.$ Also, the Lanzhou characteristic polynomial of a graph $\Gamma$ is defined by $\phi(A_{Lz}(\Gamma))= \vert \lambda I - A_{Lz}(\Gamma)\vert.$ The Lanzhou energy of graphs is defined as,
$$E_{Lz}(\Gamma)= \sum\limits_{i=1}^{n} \left\vert\lambda_i(\Gamma)\right\vert.$$
Two graphs are said to be $Lz-$equienergetic if they have the same Lanzhou energy.

Throughout the manuscript, the inertia of the adjacency matrix corresponding to graph $\Gamma$ will be called ordinary inertia and represented by $In(A(\Gamma)),$ while the inertia of the Lanzhou matrix relating to graph $\Gamma$ will be called Lanzhou inertia and represented by $In(A_{Lz}(\Gamma)).$

The first intriguing characteristic of this matrix that we observed is that the ordinary inertia and the Lanzhou inertia associated with complete graph is different.

 The Lanzhou spectrum and Lanzhou inertia of $K_n$ is given by $$\begin{pmatrix}
		0 \\
		n \\
\end{pmatrix}.$$
$$In(A_{Lz}(K_n))=(0,\ 0,\ n).$$

The spectrum and ordinary inertia of $K_n$ is given by $$\begin{pmatrix}
		n-1& -1 \\
		1&n-1 \\
\end{pmatrix}.$$
 $$In(A(K_n))=(1,\ n-1,\ 0).$$

We also found two other class of graphs whose ordinary inertia was different from Lanzhou inertia.
\begin{enumerate}
\item  The Lanzhou spectrum and inertia of removal of one edge from complete graph $K_n$ is
$$\begin{pmatrix}
0 & -(n-2)\sqrt{2(n-2)} & (n-2)\sqrt{2(n-2)}\\
n-2&1 &1\\
\end{pmatrix}$$

$$In(A_{Lz}(K_n/e))=(1,\ 1,\ n-2).$$

The adjacency spectrum and inertia of removal of one edge from complete graph $K_n$ is
$$\begin{pmatrix}
0 &-1& \dfrac{(n-3)+\sqrt{n^2+2n-7}}{2} & \dfrac{(n-3)-\sqrt{n^2+2n-7}}{2}\\
1&n-3& 1 &1\\
\end{pmatrix}$$

$$In(A(K_n/e))=(1,\ n-2,\ 1).$$

\item The Lanzhou spectrum and inertia of removal of $\left\lfloor \dfrac{n}{2}\right\rfloor-1$ independent edges from $K_n,\ n\geq 7$ whenever $n$ is odd is

$$\begin{pmatrix}
0 & -4(n-2) & A & B\\
\left\lfloor \dfrac{n}{2}\right\rfloor+1 &n- \left\lfloor \dfrac{n}{2}\right\rfloor-3 &1 &1\\
\end{pmatrix},$$
where $A=n^2-7n+10+(n-2)\sqrt{n^2-7n+16}$ and $B=n^2-7n+10-(n-2)\sqrt{n^2-7n+16}$

$$In(A_{Lz}(\Gamma))=(1,\ n- \left\lfloor \dfrac{n}{2}\right\rfloor-2,\ \left\lfloor \dfrac{n}{2}\right\rfloor+1).$$

 The adjacency spectrum and inertia of removal of $\left\lfloor \dfrac{n}{2}\right\rfloor-1$ independent edges from $K_n,\ n\geq 7$ whenever $n$ is odd is
$$\begin{pmatrix}
0 & -1&-2&A&B\\
\left\lfloor \dfrac{n}{2}\right\rfloor-1&2&n-\left\lfloor \dfrac{n}{2}\right\rfloor-3&1&1\\
\end{pmatrix},$$
where $A=\dfrac{n-3+\sqrt{n^2-2n+13}}{2}$ and $B=\dfrac{n-3-\sqrt{n^2-2n+13}}{2}.$

$$In(A(\Gamma))=(1,\ n- \left\lfloor \dfrac{n}{2}\right\rfloor,\ \left\lfloor \dfrac{n}{2}\right\rfloor-1).$$
\end{enumerate}
The Lanzhou matrix demonstrates behavior distinct from that of the traditional adjacency matrix. In this paper, we establish several noteworthy spectral properties of the Lanzhou matrix.
\section{Preliminaries}
The following are the results required to prove the results in the following sections.

\begin{lemma}\label{4}\cite{25}\textbf{Weighted Mean-Max Inequality}
$$\dfrac{p_1+p_2+\ldots+p_n}{q_1+q_2+\ldots+q_n}\leq max_i \dfrac{p_i}{q_i}.$$
\end{lemma}
\begin{lemma}\label{21}[Theorem 2.6 in \cite{23}]

If $\overline{a}=(a_1,a_2,\ldots, a_n),\ \overline{b}=(b_1,b_2,\ldots,b_n),\ \overline{c}=(c_1,c_2,\ldots,c_n)$ and

$\overline{d}=(d_1,d_2,\ldots,d_n)$ are sequences of real numbers and $\overline{p}=(p_1,p_2,\ldots,p_n),$

$\overline{q}=(q_1,q_2,\ldots,q_n)$ are nonnegative, then
$$\sum p_ia_i^2\sum q_ib_i^2+\sum p_ic_i^2\sum q_id_i^2\geq 2\sum p_ia_ic_i\sum q_ib_id_i.$$
If $\overline{p}$ and $\overline{q}$ are sequences of positive numbers, then the equality holds if and only if $a_ib_j=c_id_j$ for any $i,j\in \{1,2,\ldots, n\}.$
\end{lemma}
\begin{lemma}\cite{23}\label{22} Let $a_1,a_2, \ldots,a_n$ and $b_1,b_2,\ldots,b_n$ be real numbers. If there exist real constants $r$ and $R$ such that for each i, $i=1,2,\ldots,n,$ $ra_i \leq b_i\leq Ra_i,$ then
$$\sum\limits_{i=1}^{n} b_i^2+rR\sum\limits_{i=1}^{n} a_i^2 \leq (r+R)\sum\limits_{i=1}^{n} a_ib_i.$$

The equality holds if $ra_i=b_i=Ra_i$ for at least one $1\leq i \leq n.$
\end{lemma}

\begin{lemma}\cite{so1994commutativity}\label{4.1}
Let $A$ and $B$ be Hermitian matrices of order $n,$ and let $1\leq i \leq n$ and $1\leq j\leq n.$ Then
$$\lambda_i(A)+\lambda_n(B)\leq \lambda_i(A+B)\leq \lambda_i(A)+\lambda_1(B).$$
\end{lemma}

\begin{lemma}\cite{godsil2013algebraic}\label{4.2}
Let $M$ be Hermitian matrix of order $n.$ If $M_r$ is a principal submatrix of $M$ of order $r,$ then $$\lambda_{n-r+i}(M)\leq \lambda_i(M_r)\leq \lambda_i(M)$$ for $1\leq i\leq r.$
\end{lemma}

\begin{lemma}\cite{merikoski2003characterizations}\label{4.3}
If $M$ is a non-zero Hermitian matrix, then $$Sp(M)\geq \left\vert\dfrac{tr(M^3)}{tr(M^2)}-\dfrac{tr(M)}{n}\right \vert.$$
\end{lemma}
\begin{lemma}\cite{pachpatte2012analytic}\label{4.4}
Let $a\leq a_i\leq A$ and $b\leq b_i\leq B.$ Consider $t_i\geq 0,$ and $T=\sum\limits_{i=1}^{n}t_i.$ Then
$$\left\vert\dfrac{1}{T}\sum\limits_{i=1}^{n}t_ia_ib_i-\dfrac{1}{T^2}\sum\limits_{i=1}^{n}t_ia_i\sum\limits_{i=1}^{n}t_ib_i\right\vert \leq \dfrac{1}{4}(A-a)(B-b).$$
\end{lemma}





\begin{lemma}\label{4.5} \cite{24}
\textbf{Sylvester's law of inertia}
	
If $A$ is a symmetric matrix, then for any invertible matrix $S,$ the number of positive, negative and zero eigenvalues of $D=SAS^{T}$ is constant.
\end{lemma}

\begin{lemma}\label{4.6} \cite{17}
	Let $B=\begin{bmatrix}
		B_0 & B_1\\
		B_1 & B_0\\
	\end{bmatrix}$ be a symmetric $2\times 2$ block matrix where $B_0$ and $B_1$ are square matrices of same order. Then the spectrum of $B$ is the union of spectra of $B_0+B_1$ and $B_0-B_1.$
\end{lemma}
\begin{lemma}\label{4.7}\cite{18}
	Let $R,\ S,\ X,\ Y$ be the $n\times n$ matrices. If $R$ is invertible and $RX=XR,$ then $$\begin{vmatrix}
		R & S\\
		X & Y\\
	\end{vmatrix}=\vert RY-XS\vert. $$
\end{lemma}

\begin{lemma}\label{4.8}\cite{19}
\textbf{Haynsworth inertia additivity formula}
	
	Let $H$ be an $n\times n$ Hermitian matrix partitioned as $H= \begin{bmatrix}
		H_{11} & H_{12}\\
		H_{21} & H_{22}\\
	\end{bmatrix}.$
	If $H_{11}$ is non-singular, then $In(H)=In(H_{11})+In(H/H_{11}),$ where $H/H_{11}$ is the Schur complement of $H_{11}.$
\end{lemma}

\section{Initial findings}

\begin{theorem}
Let $G$ be a simple graph with $n$ vertices and $m$ edges. Let $P = 2 \sum\limits_{i<j} l_{ij}^2$ and $Q = 6 \sum\limits_{i<j<k} l_{ij}l_{jk}l_{ki}$, where $l_{ij}$ are the $(i,j)^{th}$ entries of the Lanzhou matrix $A_{Lz}(G).$ If the Lanzhou eigenvalues are $\lambda_1, \ldots, \lambda_n$, then:
\begin{enumerate}
\item $\sum\limits_{i=1}^{n} \lambda_i =0.$
\item $\sum\limits_{i=1}^{n} \lambda_i^2 = P.$
\item $\sum\limits_{i=1}^{n} \lambda_i^3 = Q.$
\item $\sum\limits_{1\leq i<j\leq n} \lambda_i\lambda_j = -\dfrac{P}{2}.$
\end{enumerate}
\end{theorem}
\begin{proof}
\begin{enumerate}
\item Sum of eigenvalues of $A_{Lz}(G)$ is equal to trace of $A_{Lz}(G).$
		
$$\sum \limits_{i=1}^{n} \lambda_i=\sum \limits_{i=1}^{n} l_{ii}=0.$$
		
\item Sum of squares of eigenvalues of $A_{Lz}(G)$ is the trace of $[A_{Lz}(G)]^2.$
		
\begin{align*}
\sum \limits_{i=1}^{n} \lambda_i^2&=\sum \limits_{i=1}^{n} \sum \limits_{j=1}^{n} l_{ij}l_{ji}\\&=\sum \limits_{i=1}^{n}l^2_{ii}+\sum \limits_{i\neq j} l_{ij}l_{ji}\\\sum \limits_{i=1}^{n} \lambda_i^2&=0 +2 \sum \limits_{i< j}l_{ij}^2\\&=P, \text{\ where}\ P = 2 \sum_{i<j} l_{ij}^2.
\end{align*}

\item Sum of cubes of eigenvalues of $A_{Lz}(G)$ is the trace of $[A_{Lz}(G)]^3.$
\begin{align*}
\sum \limits_{i=1}^{n} \lambda_i^3&=\sum \limits_{i=1}^{n} \sum \limits_{j=1}^{n} \sum \limits_{k=1}^{n} l_{ij}l_{jk}l_{ki}\\&=\sum \limits_{i=1}^{n}l^3_{ii}+3\sum \limits_{i\neq j} l_{ii} l_{ij}l_{ji}+\sum\limits_{i<j<k}l_{ij}l_{jk}l_{ki}\\\sum \limits_{i=1}^{n} \lambda_i^3&=6 \sum \limits_{i< j<k}l_{ij}l_{jk}l_{ki}\\&=Q, \text{\ where}\ Q = 6 \sum_{i<j<k} l_{ij}l_{jk}l_{ki}.
\end{align*}
		 
\item We have, \begin{align*}
\sum\limits_{1\leq i<j\leq n} \lambda_i\lambda_j&=\dfrac{1}{2}\left(\left(\sum\limits_{i=1}^{n}\lambda_i\right)^2- \sum\limits_{i=1}^{n}\lambda_i ^2\right)\\&=\dfrac{1}{2}\left( 0-2 \sum \limits_{i< j}l_{ij}^2\right)\\&=-\dfrac{P}{2}.
\end{align*}
\end{enumerate}
\end{proof}
\begin{theorem}\label{4.9}
	Let $G$ be a $r-$regular graph on $n$ vertices, where $r<n-1.$ Let $A(G)$ be the adjacency matrix of graph $G.$ Then $A(G)$ and $A_{Lz}(G)$ have the same number of positive, negative and zero eigenvalues.
\end{theorem}
\begin{proof}
	Let $G$ be a $r-$regular graph, where $r\neq n-1.$ Then $A_{Lz}(G)=D^{\frac{1}{2}}A(G)D^{\frac{1}{2}},$ where $D$ is the diagonal matrix with diagonal entry as $2r(n-r-1).$ 
	
	According to Sylvester's law of inertia [\ref{4.5}], $A(G)$ and $A_{Lz}(G)$ have the same number of positive, negative and zero eigenvalues.
\end{proof}
\begin{theorem}\label{4.10}
Let $G$ be a graph of order $n.$ Then $\lambda_1=\lambda_2=\cdots=\lambda_n=0$ if and only if $G=K_n$ or $G=\overline{K_n}.$
\end{theorem}
\begin{proof}
Let $G$ be a graph of order $n.$ Let $\lambda_1=\lambda_2=\cdots=\lambda_n=0.$ This implies the rank of $A_{Lz}(G)$ is zero. Suppose $G$ is connected and not complete; then there exist at least two vertices which are not of full degree. This implies the rank of $A_{Lz}(G)$ is at least two, a contradiction. Similarly, if $G$ is not $\overline{K_n},$ then again the rank will be at least two, a contradiction.

Therefore, $G=K_n$ or $G =\overline{K_n}.$ Converse holds trivially.
\end{proof}
\begin{theorem}
	Let $G$ be a graph of order $n.$ Then $\vert \lambda_1\vert=\vert\lambda_2\vert=\cdots=\vert\lambda_n\vert$ if and only if $G=pK_2$ or $G=K_n$ or $G=\overline{K_n}.$ 
\end{theorem}
\begin{proof}
	Let $G$ be a graph of order $n.$ Suppose $\vert \lambda_1\vert=\vert\lambda_2\vert=\cdots=\vert\lambda_n\vert.$

	For the case when $G=K_n$ or $G=\overline{K_n},$ proof follows from Theorem [\ref{4.10}].

	To prove $G=pK_2.$ Let $H$ be one of the components in $G$ such that there exists a vertex $v$ whose degree is greater than or equal to two. Let $\mu_1\geq \mu_2\geq \cdots\geq \mu_k,$ where $k\leq n$ are the Lanzhou eigenvalues of $H.$ By Cauchy's interlacing theorem, $\lambda_1\geq \mu_1\geq \lambda_2\geq \mu_2\geq \cdots \geq \lambda_n.$ Since $\lambda_i$'s are equal, by the Perron-Frobenius theorem $\lambda_1(G)> \mu_2(H)$ does not hold. Therefore, $G=pK_2.$ Converse is trivial.
\end{proof}
The following proposition follows from the definition of Lanzhou matrix and Theorem [\ref{4.9}].
\begin{proposition}
Let $G$ be a $r-$regular graph of order $n.$ Then $$det(A_{Lz}(G))=(2r(n-r-1))^n det(A(G)).$$
\end{proposition}
Let $A_{ij}$ and $Lz_{ij}$ denote the $(i,j)^{th}$ entry in the adjacency and the Lanzhou matrix corresponding to graph $G,$ respectively.
\begin{lemma}\label{4.11}
Let $G$ be a graph of order $n$ in which there is at most one vertex of degree $n-1.$ Then $A_{ij}\neq 0$ if and only if $Lz_{ij}\neq 0.$
\end{lemma}
\begin{proof}
Let $A_{ij}\neq 0$ imply $v_i\sim v_j\in E(G).$ Without loss of generality, let $d(v_i)=a$ and $d(v_j)=b.$ Then we have the following three cases.
\begin{enumerate}
\item If $a=n-1,\ 0< b<n-1,$ then since $n\neq b$ and $b\neq 0,$ $Lz_{ij}\neq 0.$
\item Similar argument holds when $a<n-1$ and $b=n-1.$
\item If $a<n-1$ and $b<n-1,$ then again $Lz_{ij}\neq 0.$
\end{enumerate}
Therefore, $A_{ij}\neq 0$ implies $Lz_{ij}\neq 0.$

Conversely, since $G$ has at most one vertex of degree $n-1,$ $Lz_{ij}\neq 0$ implies $A_{ij}\neq 0,$ for all $v_i,v_j\in V(G).$
\end{proof}
\begin{lemma}
Suppose $G$ is an $n-$vertex graph. Then $\mathcal{E}_{Lz}(G)=0$ if and only if $G=K_n$ or $G=\overline{K_n}.$
\end{lemma}
\begin{proof}
The lemma follows directly from the definition of Lanzhou matrix.
\end{proof}
\section{Bounds on spread of the Lanzhou matrix}\label{v}
\begin{remark}
Let $G$ be a $r-$regular graph, $r<n-1.$ 
We know that $\lambda_1(A_{Lz}(G))=2r(n-r-1)\lambda_1(A(G))$
and $\lambda_n(A_{Lz}(G))=2r(n-r-1)\lambda_n(A(G)).$
Therefore $$S(A_{Lz}(G))=2r(n-r-1)S(A(G)).$$
\end{remark}
\begin{theorem}
Let $S(A_{Lz}(G))$ be the spread of the Lanzhou matrix of graph $G.$ Then $$S(A_{Lz}(G))\geq \sqrt{\dfrac{4P}{n}}.$$

\end{theorem}
\begin{proof}
On substituting $t_i=1,$ $a_i=b_i=\lambda_i,$ in Lemma [\ref{4.4}], we have $T=\sum\limits_{i=}^{n} t_i=n.$ 

Therefore,
\begin{align*}
\left\vert\dfrac{1}{T}\sum\limits_{i=1}^{n}\lambda_i^2-\dfrac{1}{T^2}\left(\sum\limits_{i=1}^{n}\lambda_i\right)^2\right\vert & \leq \dfrac{1}{4}(\lambda_1-\lambda_n)^2\\\implies \left\vert \dfrac{P}{n}\right\vert &\leq \dfrac{1}{4}S(A_{Lz}(G))^2\\\implies  S(A_{Lz}(G))&\geq \sqrt{\dfrac{4P}{n}}.
\end{align*}
Next, $$\left\vert\dfrac{1}{T}\sum\limits_{i=1}^{n}\lambda_i^2-\dfrac{1}{T^2}\left(\sum\limits_{i=1}^{n}\lambda_i\right)^2\right\vert  \leq \dfrac{1}{4}(\lambda_1-\lambda_n)^2.$$

Since $\sum\limits_{i=1}^{n}\lambda_i=0,$ we have
$$ \left\vert\dfrac{1}{n}\sum\limits_{i=1}^{n}\lambda_i^2\right\vert  \leq \dfrac{1}{4}(\lambda_1-\lambda_n)^2.$$

\end{proof}

\begin{theorem}
Let $G$ be a graph on $n$ vertices, and let $u$ be a pendant vertex in $G.$ Then following inequality holds.
$$S(A_{Lz}(G\setminus u))-2\sqrt{\mathcal{A}^2+(Lz_{vw})^2}\leq S(A_{Lz}(G))\leq S(A_{Lz}(G\setminus u))+2\sqrt{\mathcal{A}^2+(Lz_{vw})^2}.$$
Here, $\mathcal{A}=d(u)d(\overline{u})+d(v)d(\overline{v}),$ and $w\in V(G)\setminus \{u,v\}.$
\end{theorem}
\begin{proof}
Let $u\in V(G)$ be a pendant vertex of $G$ and $u\sim v$ in $G.$ Label the vertex set of $G$ in such a way that $A_{Lz}(G)=A_{Lz}(G\setminus u)+\mathcal{X},$ where 
\small
$A_{Lz}(G)=\begin{bmatrix}
0 & \mathcal{A} &0_{1\times n-2} \\
\mathcal{A} & 0 & \mathcal{B}_{1 \times n-2}\\
0_{n-2\times 1} & \mathcal{B}_{n-2\times 1} & \mathcal{C}_{n-2\times n-2}\\
\end{bmatrix},$  $A_{Lz}(G\setminus u)=\begin{bmatrix}
0 & 0 &0_{1\times n-2} \\
0 & 0 & \mathcal{D}_{1 \times n-2}\\
0_{n-2\times 1} & \mathcal{D}_{n-2\times 1} & \mathcal{C}_{n-2\times n-2}\\
\end{bmatrix},$ and  $\mathcal{X}=\begin{bmatrix}
0 & \mathcal{A} &0_{1\times n-2} \\
\mathcal{A} & 0 & \mathcal{E}_{1 \times n-2}\\
0_{n-2\times 1} & \mathcal{E}_{n-2\times 1} & 0_{n-2\times n-2}\\
\end{bmatrix}.$ 
	
	\normalsize
	
	\vspace{0.5cm}

	Here, $\mathcal{A}=d(u)d(\overline{u})+d(v)d(\overline{v}),\ \mathcal{B}=d(w)d(\overline{w})+d(v)d(\overline{v}),$ for all $w\in V(G)\setminus \{u,v\}.$  $\mathcal{D}=(d(v)-1)(d(\overline{v})-1)+d(w)d(\overline{w}),$ for all $w\in V(G)\setminus \{u,v\}.$ 	
	 We define $\mathcal{E}_{1\times n-2}=Lz_{vw},$ where $Lz_{vw}=-1$ if $v\sim w,$ for all $w\in V(G)\setminus \{u,v\}.$ 
	
	Now clearly, the Lanzhou eigenvalues of $\mathcal{X}$ are $\sqrt{\mathcal{A}^2+(Lz_{vw})^2},$ $-\sqrt{\mathcal{A}^2+(Lz_{vw})^2},$ and $0$ with multiplicities $1,\ 1,$ and $n-2,$ respectively.

	By Lemma [\ref{4.1}], we have $$\lambda_1(A_{Lz}(G\setminus u))-\sqrt{\mathcal{A}^2+(Lz_{vw})^2}\leq \lambda_1(A_{Lz}(G))\leq \lambda_1(A_{Lz}(G\setminus u))+\sqrt{\mathcal{A}^2+(Lz_{vw})^2}$$ and  
	$$\lambda_n(A_{Lz}(G\setminus u))-\sqrt{\mathcal{A}^2+(Lz_{vw})^2}\leq \lambda_n(A_{Lz}(G)\leq \lambda_n(A_{Lz}(G-u))+\sqrt{\mathcal{A}^2+(Lz_{vw})^2}.$$
	Thus, $$S(A_{Lz}(G))\leq S(A_{Lz}(G\setminus u))+2\sqrt{\mathcal{A}^2+(Lz_{vw})^2}\text{\ and }$$ 
	
	$$S(A_{Lz}(G))\geq S(A_{Lz}(G\setminus u))-2\sqrt{\mathcal{A}^2+(Lz_{vw})^2}.$$

	Therefore, $$S(A_{Lz}(G\setminus u))-2\sqrt{\mathcal{A}^2+(Lz_{vw})^2}\leq S(A_{Lz}(G))\leq S(A_{Lz}(G\setminus u))+2\sqrt{\mathcal{A}^2+(Lz_{vw})^2}.$$
\end{proof}
\begin{theorem}
	Let $T$ be a tree on $n$ vertices and $A_{Lz}(T_u)$ be the principal submatrix of $T$ obtained by deleting the rows and columns corresponding to $u.$ Then for $\beta>0,$ we have
	$$S(A_{Lz}(T\setminus u))\leq S(A_{Lz}(T_u))+\beta\leq S(A_{Lz}(T))+2\beta.$$
\end{theorem}
\begin{proof}
Let $T$ be a tree on $n$ vertices. Clearly, the Lanzhou eigenvalues of $T$ are symmetric about the origin, see [\ref{4.19}]. Let $u\in V(T)$ and $A_{Lz}(T_u)$ be the principal submatrix of $T$ obtained by deleting the rows and columns corresponding to $u.$ We define $\mathcal{X}=A_{Lz}(T_u)-A_{Lz}(T\setminus u),$ where in $\mathcal{X}$ is a matrix whose entries are $0$ and $\beta.$

In Lemma [\ref{4.1}], when $i=1,$ 
$$\lambda_1(A_{Lz}(T\setminus u))-\beta\leq \lambda_1(A_{Lz}(T_u))\leq \lambda_1(A_{Lz}(T\setminus u))+\beta.$$
When $i=n-1,$  $$\lambda_{n-1}(A_{Lz}(T\setminus u))-\beta\leq \lambda_{n-1}(A_{Lz}(T_u))\leq \lambda_{n-1}(A_{Lz}(T\setminus u))+\beta.$$

Next, from Lemma [\ref{4.2}], $$\lambda_1(A_{Lz}(T\setminus u))\leq \lambda_1(A_{Lz}(T_u))\leq \lambda_1(A_{Lz}(T))$$ and $$\lambda_n(A_{Lz}(T))\leq \lambda_{n-1}(A_{Lz}(T_u))\leq \lambda_{n-1}(A_{Lz}(T\setminus u)).$$

Thus, $$S(A_{Lz}(T-u))\leq S(A_{Lz}(T_u))+\beta\leq S(A_{Lz}(T))+2\beta.$$
\end{proof}
\begin{theorem}
Let $\vert\vert A_{Lz}(G)\vert \vert_F$ denote the Frobenius norm of $A_{Lz}(G).$ Then 
$$S(A_{Lz}(G))\leq \lambda_1+\sqrt{\vert\vert A_{Lz}(G)\vert \vert_F^2-\lambda_1^2}.$$ Equality holds if and only if $\lambda_i=0,$ for all $2\leq i\leq n-1$ and $\lambda_1=-\lambda_n.$
\end{theorem}
\begin{proof}
We consider the Frobenius norm of the Lanzhou matrix of a graph $G,$ denoted by $\vert\vert A_{Lz}(G)\vert \vert_F,$ which is given by 
	
$$\vert\vert A_{Lz}(G)\vert \vert_F=\sqrt{\sum\limits_{i=1}^{n}\sum\limits_{j=1}^{n}\vert l_{ij}\vert^2}=\sqrt{tr(A_{Lz}(G)^TA_{Lz}(G))}=\sqrt{\sum\limits_{i=1}^{n}\lambda_i^2(A_{Lz}(G))}.$$
	
Clearly, $\vert\vert A_{Lz}(G)\vert \vert_F^2=tr(A_{Lz}(G))^2=\sum\limits_{i=1}^{n}\lambda_i^2.$
	
This implies, $\lambda_1^2+\lambda_n^2\leq \vert\vert A_{Lz}(G)\vert \vert_F^2.$
	
Therefore, we have

$$-\sqrt{\vert\vert A_{Lz}(G)\vert \vert_F^2-\lambda_1^2}\leq \lambda_n \leq \sqrt{\vert\vert A_{Lz}(G)\vert \vert_F^2-\lambda_1^2}.$$
	
So, $$S(A_{Lz}(G))\leq \lambda_1+\sqrt{\vert\vert A_{Lz}(G)\vert \vert_F^2-\lambda_1^2}.$$

Next, consider the following:
\begin{align*}
\lambda_1 - \lambda_n &\leq \lambda_1 + \sqrt{\|A_{Lz}(G)\|_F^2 - \lambda_1^2} \\
&\implies -\lambda_n \leq \sqrt{\|A_{Lz}(G)\|_F^2 - \lambda_1^2} \\
&\implies \|A_{Lz}(G)\|_F^2 \leq \lambda_1^2 + \lambda_n^2.
\end{align*}
Since
\[
\sum_{i=2}^{n-1} \lambda_i^2 + \lambda_1^2 + \lambda_n^2 = \lambda_1^2 + \lambda_n^2,
\]
equality holds if and only if
\[
\sum_{i=2}^{n-1} \lambda_i^2 = 0 \iff \lambda_i = 0, \text{\ for all\ } 2 \leq i \leq n-1.
\]

In the non-trivial case, since $\sum_{i=1}^n \lambda_i = 0$, $\sum_{i=1}^n \lambda_i^2 = \lambda_1^2 + \lambda_n^2$, and $\lambda_i = 0$ for all $2 \leq i \leq n-1$, it follows that $\lambda_1 = -\lambda_n$. Thus, the eigenvalues of $A_{Lz}(G)$ are symmetric about the origin, and $\operatorname{rank}(A_{Lz}(G)) = n-2$.

By Theorem [\ref{4.19}], there exists a non-trivial partition $(X,Y)$ of $V(\Gamma)$ such that $\Gamma$ is one of the following:
\begin{enumerate}
\item $V(X)$ induces a clique with all vertices of full degree, and $V(Y)$ induces an independent set.
\item $V(Y)$ induces a clique with all vertices of full degree, and $V(X)$ induces an independent set.
\item Both $V(X)$ and $V(Y)$ induce independent sets; if edges exist between $X$ and $Y,$ then $\Gamma \cong K_{m,n}.$
\end{enumerate}
\end{proof}

\begin{theorem}
Let $\vert\vert A_{Lz}(G)\vert \vert_F$ denote the Frobenius norm of $A_{Lz}(G).$ Then 
$$S(A_{Lz}(G))\geq \lambda_1+\sqrt{\dfrac{\vert\vert A_{Lz}(G)\vert \vert_F^2-k\lambda_1^2}{n-k}}.$$
\end{theorem}

\begin{proof}
	We know that $\lambda_1\geq \lambda_2\geq \cdots\geq \lambda_k>0\geq \lambda_{k+1}\geq \cdots \geq \lambda_n .$
	
	So, $\lambda_i^2\leq \lambda_1^2,$ for $1\leq i\leq k$ and $\lambda_i^2\leq \lambda_n^2,$ for $k+1\leq i\leq n.$
	
	Now, we have $$\vert\vert A_{Lz}(G)\vert \vert_F^2\leq k\lambda_1^2+(n-k)\lambda_n^2.$$
	
	$$\implies -\sqrt{\dfrac{\vert\vert A_{Lz}(G)\vert \vert_F^2-k\lambda_1^2}{n-k}}\geq \lambda_n\geq \sqrt{\dfrac{\vert\vert A_{Lz}(G)\vert \vert_F^2-k\lambda_1^2}{n-k}}.$$
	
	Therefore, spread of $A_{Lz}(G)$ is 
	$$S(A_{Lz}(G))=\lambda_1-\lambda_n\geq \lambda_1+\sqrt{\dfrac{\vert\vert A_{Lz}(G)\vert \vert_F^2-k\lambda_1^2}{n-k}}.$$
\end{proof}

\section{Bounds on Lanzhou energy of a graph}\label{4.12}
\begin{theorem}
Let $G$ be a graph of order $n$ and $P=2\sum\limits_{i< j}{l_{ij}}^2.$ Then
$$\mathcal{E}_{Lz}(G)\leq \sqrt{nP}.$$
\end{theorem}
\begin{proof}
We have 
\begin{align*}
\sum\limits_{i=1}^{n}\sum\limits_{j=1}^{n}(\vert \lambda_i-\lambda_j\vert)^2&\geq 0\\\implies n\sum\limits_{i=1}^{n}\vert \lambda_i\vert^2+n	\sum\limits_{j=1}^{n}\vert \lambda_j\vert ^2-2\sum\limits_{i=1}^{n}\vert \lambda_i\vert \sum\limits_{j=1}^{n}\vert \lambda_j\vert &\geq 0\\\implies 2nP&\geq 2\mathcal{E}_{Lz}(G)^2\\\implies \mathcal{E}_{Lz}(G)&\leq \sqrt{nP}.
\end{align*}
\end{proof}
\begin{theorem}
Let $G$ be a graph of order $n$ and let $P=2\sum\limits_{i< j}{l_{ij}}^2.$ If $\lambda_1\geq \lambda_2\geq \cdots\geq \lambda_n$ are the Lanzhou eigenvalues of graph $G,$ then
$$\lambda_1\leq \sqrt{\dfrac{(n-1)P}{n}}.$$
\end{theorem}
\begin{proof}
On applying Cauchy Schwarz theorem for $i=2,\ldots,n$ and $a_i=\lambda_i$ and $b_i=1,$ we obtain 
\begin{align}\label{4.13}
\left(\sum\limits_{i=2}^{n}\lambda_i\right)^2\leq \sum\limits_{i=2}^{n}\lambda_i^2\sum\limits_{i=2}^{n}1.
\end{align}
We have $\sum\limits_{i=1}^{n}\lambda_i=0\implies\sum\limits_{i=2}^{n}\lambda_i=-\lambda_1$ and $\sum\limits_{i=2}^{n}\lambda_i^2=P-\lambda_1^2.$
	
On substituting in Equation [\ref{4.13}], we obtain
\begin{align*}
(-\lambda_1)^2&\leq (n-1)(P-\lambda_1^2)\\\implies \dfrac{1}{n-1}&\leq \dfrac{P-\lambda_1^2}{\lambda_1^2}\\\implies  \lambda_1^2&\leq \dfrac{(n-1)P}{n}\\\implies \lambda_1&\leq \sqrt{\dfrac{(n-1)P}{n}}.
\end{align*}
\end{proof}
\begin{theorem}
Let $G$ be a graph of order $n$ and $\rho(A_{Lz}(G))$ be the Lanzhou spectral radius of graph $G.$ Then
$$\mathcal{E}_{Lz}(G)\leq n\rho(A_{Lz}(G)).$$
\end{theorem}
\begin{proof}
On substituting $p_i=\vert \lambda_i\vert$ and $q_i=1,$ for all $i\in V(G),$ in Lemma [\ref{4}], we obtain \begin{align*}
\dfrac{\vert \lambda_1\vert+\vert \lambda_2\vert+\cdots+\vert\lambda_n\vert}{1+1+\cdots+1}&\leq max_{i}\vert\lambda_i\vert\\\implies \dfrac{\sum\limits_{i=1}^{n}\vert\lambda_i\vert}{n}&\leq \vert \lambda_1\vert \\\implies \mathcal{E}_{Lz}(G)&\leq n\rho(A_{Lz}(G)).
\end{align*}
\end{proof}

\subsection{Observations}
\begin{enumerate}
	\item Let $G$ be a disconnected graph. Let $G'$ be a graph obtained by adding exactly one edge to $G$ such that $G$ becomes connected. Then the Lanzhou energy of $G'$ is greater than the Lanzhou energy of $G.$ 
	
	Illustrations for the above mentioned observation are shown in Figure [\ref{4.14}], Figure [\ref{4.15}], and Figure [\ref{4.16}].
	
	\item Let $G$ be a connected $n$ vertex graph in which every vertex has degree strictly less than $n-1.$ Let $G'$ be a connected $n$ vertex graph in which at least one vertex has degree $n-1.$ Then $\mathcal{E}_{Lz}(G)\geq \mathcal{E}_{Lz}(G').$
	
	\item Whenever $G$ is $r-$regular, $r<n-1,$ $A_{Lz}(G)=2r(n-r-1)A(G).$ Therefore, all the simple graphs for which $E(G)=E(\overline{G}),$ $\mathcal{E}_{Lz}(G)=\mathcal{E}_{Lz}(\overline{G})$ holds. Some graphs for which $E(G)=E(\overline{G})$ are listed in \cite{26}.
	
	\item The smallest Lanzhou equienergetic graph pairs are $C_4$ and $\overline{C_4}$, with $\mathcal{E}_{Lz}(C_4)=\mathcal{E}_{Lz}(\overline{C_4})=16$ and Lanzhou eigenvalues of $C_4$ are $-8,\ 8,\ 0,\ 0$, and Lanzhou eigenvalues of $\overline{C_4}$ are $-4,\ -4,\ 4,\ 4.$
	\begin{figure}[H]
		\centering
			\begin{tikzpicture}
				[scale= 1.0, auto= center, every node/.style = {circle,fill=black!100}]
				\node (v0) at (-2,0) {};
				\node (v1) at (0,0) {};
				\node (v2) at (2,0) {};
				\node (v3) at (4,0) {};
				\draw (v1) -- (v2);
				\draw (v2) -- (v3);
			\end{tikzpicture}
			\caption{Graph $G$}\label{4.14}
		\end{figure}
		
		\vspace{1.0cm}
		\begin{figure}[H]
			\centering
			\begin{tikzpicture}
				[scale= 1.0, auto= center, every node/.style = {circle,fill=black!100}]
				\node (v0) at (-2,0) {};
				\node (v1) at (0,0) {};
				\node (v2) at (2,0) {};
				\node (v3) at (4,0) {};
				\draw (v1) -- (v2);
				\draw (v2) -- (v3);
				\draw (v1) -- (v0);
			\end{tikzpicture}
			\caption{Graph $G'$}\label{4.15}
		\end{figure}
		
		\vspace{1.0cm}
		\begin{figure}[H]
			\centering
			\begin{tikzpicture}
				[scale= 1.0, auto= center, every node/.style = {circle,fill=black!100}]
				\node (v0) at (-2,0) {};
				\node (v1) at (0,0) {};
				\node (v2) at (2,0) {};
				\node (v3) at (4,0) {};
				\draw (v1)--(v2);
				\draw (v2)--(v3);
				\draw (v1)--(v0);
				\draw (v3)--(v0);
				\draw[ultra thick, -] (4,0) arc (0:180:3);
			\end{tikzpicture}
			\caption{Graph $G''$}\label{4.16}
	\end{figure}

	The Lanzhou energy of $G,\ G',\ G''$ are given by $11.3137,\ 17.8885,\ 16,$ respectively.

\end{enumerate}

\section{The Lanzhou spectrum of some graphs}\label{4.17}
\begin{theorem}	
	For a cycle $C_{2n},$ the Lanzhou eigenvalues are $2a,\  -2a,$ and $2acos\dfrac{\pi m}{n},$ where $m=1, \ldots, 2n-1$ and $a=4(2n-3).$
\end{theorem}
\begin{proof}
	The Lanzhou matrix of $C_{2n}$ is $$A_{Lz}(C_{2n})=\begin{bmatrix}
		0 & a & 0 &\cdots &0 &a\\
		a & 0 & a &\cdots & 0 & 0\\
		0 & a &0 & \cdots & 0 & 0\\
		\vdots & \vdots & \vdots & \ddots & \vdots& \vdots\\
		a & 0 &0 & \cdots & a & 0\\
	\end{bmatrix}_{2n}.$$
Clearly, $A_{Lz}(C_{2n})$ is a circulant matrix of order $2n,$ where $a_i=0,$ for $i=1,3,\ldots, 2n-1$ and $a_2=a_{2n}=a=4(2n-3).$ 
	
As proved in \cite{17}, the eigenvalues of a circulant matrix are given by	$$ \lambda_m=\sum\limits_{k=1}^{n}a_ke^{\frac{2\pi i m(k-1)}{n}},\ m=0,1,\ldots,(n-1).$$
	
	Let $\omega=e^{\frac{2\pi i}{n}}.$ Then $$ \lambda_m=\sum\limits_{k=1}^{n}a_k\omega^{m(k-1)}.$$
	
	As $A_{Lz}(C_{2n})$ is of order $2n,$ $\omega =e^{\frac{\pi i}{n}},\ \omega^{n}=-1$ and $\lambda_m=a_1\omega^{0m}+a_2\omega^{1m}+a_3\omega^{2m}+\cdots+a_{n+1}\omega^{nm}+\cdots+a_{2n}\omega^{(2n-1)m},\ m=0,\ldots, (2n-1),$ we have $a_1=a_3=a_4=a_5=\cdots=a_{2n-1}=0,$ $a_2=a_{2n}=a=4(2n-3).$ 
	
	Hence,
	$$\lambda_m=a\omega^m+a\omega^{(2n-1)m}.$$ 
	
	Therefore,
	$$\lambda_m =
	\begin{cases} 2a, &\text{if $m=0$ } \\
		-2a, &\text{if $m=n$}\\
		2acos\left(\dfrac{\pi m}{n}\right), &\text{$0<m<n,\ n<m\leq 2n-1$}.
	\end{cases}$$
	Consequently, $\vert A_{Lz}(C_{2n})\vert =-2^{2n}a^{2n}\prod\limits_{\substack{m=1\\m\neq n}}^{2n-1}cos\left(\dfrac{\pi m}{n}\right).$
\end{proof}
\begin{theorem}
	The Lanzhou spectrum of cycle of order $2n+1$ is $$\begin{pmatrix}
		8(2n-2) & 8(2n-2)\dfrac{cos2\pi}{2n+1}& \cdots & 8(2n-2)\dfrac{cos(2n)\pi}{2n+1}\\
		1 & 2& \cdots & 2\\
	\end{pmatrix}.$$
\end{theorem}
\begin{proof}
	The proof directly follows from the fact that $A_{Lz}(C_{2n+1})=4(2n-2)A(C_{2n+1}),$ where $A(C_{2n+1})$ is the adjacency matrix of $C_{2n+1}.$
\end{proof}

\begin{theorem}
	The Lanzhou spectrum of $S_n^0$ is $$\begin{pmatrix}
		-a & a(n-1) & a & -a(n-1)\\
		n-1 & 1 & n-1 & 1\\
	\end{pmatrix}.$$ Also, $In(A_{Lz}(S_n^0))= (n,n,0),$ for all $n\geq 2.$
\end{theorem}
\begin{proof}
	The Lanzhou matrix of Crown graph $S_n^0$ is given by
	$$A_{Lz}(S_n^0)=\begin{bmatrix}
		0_n & [a(J-I)]_n\\
		[a(J-I)]_n & 0_n\\
	\end{bmatrix}_{2n},$$ where $a=2n(n-1).$
	
	By Lemma [\ref{4.6}], spectrum of $S_n^0$ is the union of spectra of $[a(J-I)]_n$ and spectra of $-[a(J-I)]_n.$

	Now, we shall find spectra of $[a(J-I)]_n.$

	Let $B_1=\begin{bmatrix}
		a(J-I)
	\end{bmatrix}_{n}.$ Then $\vert B_1-\lambda I\vert =\begin{vmatrix}
		a(J-I)-\lambda I
	\end{vmatrix}_{n}.$

	By row and column operation on $\vert B_1-\lambda I \vert $ we obtain $$(-\lambda-a)^{n-2}(\lambda^2+a(2-n)\lambda+a^2(1-n)).$$

	Therefore, the Lanzhou eigenvalues of $S_n^0$ are $-a $ and $a(n-1)$ with multiplicity of $n-1$ and $1,$ respectively.

	By Lemma [\ref{4.6}], the Lanzhou spectrum of $S_n^0$ is given by $$\begin{pmatrix}
	-a & a(n-1) & a & -a(n-1)\\
	n-1 & 1 & n-1 & 1\\
	\end{pmatrix}.$$

	Clearly, inertia of $A_{Lz}(S_n^0)$ is $(n,n,0),$ for all $n\geq 2.$
\end{proof}

\begin{theorem}
The Lanzhou spectrum of barbell graph of order $2n$ is 
$$\begin{pmatrix}
-X & 2n-4\\
\dfrac{(n-1)X+X\sqrt{n^2-2n+5}}{2} &1\\
\dfrac{(n-1)X-X\sqrt{n^2-2n+5}}{2} & 1\\
\dfrac{(n-1)X+X\sqrt{n^2-2n+3}}{2} & 1\\
\dfrac{(n-1)X-X\sqrt{n^2-2n+3}}{2} & 1\\
\end{pmatrix},$$
	
	where $X=2n(n-1).$
\end{theorem}
\begin{proof}
	The Lanzhou matrix of barbell graph is of the form $$A_{Lz}(B)=\begin{bmatrix}
		[XJ-I(X+\lambda)]_{n} &Y_{n}\\
		Y_{n} & [XJ-I(X+\lambda)]_{n}\\
	\end{bmatrix},$$ where $J$ is the matrix of all one's and $Y=\begin{bmatrix}
		0&0&\cdots&0&0\\
		0&0&\cdots&0&0\\
		\vdots&\vdots&\ddots&\vdots&\vdots\\
		0&0&\cdots&0&0\\
		0&0&\cdots&0&X\\
	\end{bmatrix}_{n}.$

	Here, $X=2n(n-1).$ The proof follows from Lemma [\ref{4.6}].
\end{proof}
\begin{theorem}
The Lanzhou spectrum of $K_{n\times 2}$ is  $$\begin{pmatrix}
0& -4(n-1) & 8(n-1)^2 \\
n & n-1 & 1\\
\end{pmatrix}.$$ 
Also, $In(A_{Lz}(K_{n\times 2}))=(1,n-1,n).$
\end{theorem}
\begin{proof}
	The Lanzhou matrix of $K_{n\times 2}$ is given by 
	$$ A_{Lz}(K_{n\times 2})=\begin{bmatrix}
		[a(J-I)]_{n} & [a(J-I)]_{n} \\
		[a(J-I)]_{n} & [a(J-I)]_{n}\\
	\end{bmatrix}_{2n},$$
	where $a=4(n-1).$
	We have,
	$$ \vert A_{Lz}(K_{n\times 2})-\lambda I\vert =\begin{vmatrix}
		[a(J-I)]-\lambda I & [a(J-I)] \\
		[a(J-I)] & [a(J-I)]-\lambda I\\
	\end{vmatrix}_{2n}.$$
	Since $([a(J-I)]-\lambda I )\cdot [a(J-I)]= [a(J-I)]\cdot [a(J-I)]-\lambda I,$ by Lemma [\ref{4.7}], we have
	\begin{align*}
		\vert A_{Lz}(K_{n\times 2})-\lambda I\vert &=\vert ([a(J-I)]-\lambda I)^2-a^2(J-I)\vert \\&=\vert \lambda^2I-2a\lambda (J-I)\vert\\&=\vert (\lambda^2+2a\lambda)I-2a\lambda J\vert \\&=\lambda^n(\lambda+2a)^{n-2} \{\lambda^2+2a\lambda(2-n)-4a^2(n-1)\}.
	\end{align*}
	Now, the Lanzhou spectrum of $K_{n\times 2}$ is given by $$\begin{pmatrix}
		0 & -2a & 2a(n-1)\\
		n & n-1 & 1 \\
	\end{pmatrix}.$$
The number of positive, negative, and zero eigenvalues of $A_{Lz}(K_{n\times 2})$ are $1,\ n-1,\ n,$ respectively.
Therefore, $In(A_{Lz}(K_{n\times 2}))=(1,n-1,n).$
\end{proof}
\begin{theorem}\label{4.18}
If $G_1, G_2, \ldots, G_k$ are the $k\leq n$ components of a graph $G,$ then 
	
$$\mathcal{E}_{Lz}(G)=\sum\limits_{i=1}^{k}\mathcal{E}_{Lz}(G_i).$$
\end{theorem}
\begin{proof}
Let $G$ be a graph with $G_1, G_2, \ldots, G_k,$ $k\leq n$ as its components.
	
We have $G=G_1\cup G_2 \cup \cdots \cup G_k.$ Then $A_{Lz}(G)$ is a block diagonal matrix whose diagonal entries are blocks $A_{Lz}(G_1), \ldots, A_{Lz}(G_k).$
	
Then the Lanzhou spectrum of $G$ is the union of the Lanzhou spectrum of all of $G'$s components. Therefore, $\mathcal{E}_{Lz}(G)=\sum\limits_{i=1}^{k}\mathcal{E}_{Lz}(G_i).$
\end{proof}

\begin{theorem}
	The spectrum of coalescence of $k$ copies of $K_n$ is $$\begin{pmatrix}
		-2(n-1)^2 & nk-2k & A & B\\
		nk-2k & k-1 & 1 & 1 \\
	\end{pmatrix},$$ where $A=2(n-1)^2(n-2)+\dfrac{ \sqrt{2(n-1)^2(2-n)-4k(n-1)^2(1-n)}}{2}$ and
	
	$B=2(n-1)^2(n-2)-\dfrac{ \sqrt{2(n-1)^2(2-n)-4k(n-1)^2(1-n)}}{2}.$
\end{theorem}
\begin{proof}
	Let $G$ be coalescence of $k$ copies of $K_n.$
	Then the Lanzhou matrix of $G$ is given by $$A_{Lz}(G)=\footnotesize\begin{+bmatrix} [
		colspec={@{~}ccccccccc@{~}},
		colsep=0.1pt
		]
		2(n-1)^2[J-I]_{n-1} & 0_{n-1} & \cdots &0_{n-1} & (n-1)^2_{n-1\times 1}\\
		0_{n-1} &2(n-1)^2[J-I]_{n-1} & \cdots &0_{n-1} & (n-1)^2_{n-1\times 1}\\
		\vdots & \vdots& \ddots & \vdots &\vdots\\
		0_{n-1} & 0_{n-1} & \vdots & 2(n-1)^2[J-I]_{n-1} & (n-1)^2_{n-1\times 1}\\
		(n-1)^2_{n-1\times 1} & (n-1)^2_{n-1\times 1}& \cdots &(n-1)^2_{n-1\times 1}& 0_1\\
	\end{+bmatrix}_{nk-k+1}.$$
	For each block of order $n-1,$ apply row operations and then column operations as follows:
	$R_i\longrightarrow R_i-R_{i-1},\ 2\leq i\leq n-1$ and $C_i\longrightarrow C_1+C_2+\cdots+C_i,\ 1\leq i\leq n-1.$ Let $B$ be the matrix obtained after applying these operations. 
	
	\scriptsize
	$$B=(-\lambda-2(n-1)^2)^{(kn-2k)}\footnotesize\begin{+vmatrix} [
		colspec={@{~}ccccccccc@{~}},
		colsep=0.1pt
		]
		2(n-2)(n-1)^2-\lambda & 0 & \cdots & 0 & (n-1)^2\\
		0& 2(n-2)(n-1)^2-\lambda &  \cdots & 0 & (n-1)^2\\
		\vdots & \vdots &  \ddots & \vdots & \vdots\\
		0 & 0& \cdots & 2(n-2)(n-1)^2 &  (n-1)^2\\
		(n-1)^2 & (n-1)^2 &  \cdots & (n-1)^2 & (n-1)^2\\
	\end{+vmatrix}_{k+1}$$ 
	
	\normalsize
	
	The Lanzhou characteristic polynomial is given by
	\small
	$$(-\lambda-2(n-1)^2)^{nk-2k} (-\lambda-4(n-1)^2+2n(n-1)^2)^{k-1} (\lambda^2+2(n-1)^2(2-n)\lambda+k(n-1)^2(1-n)).$$
	
	\normalsize
	
	Therefore, the spectrum of coalescence of $k$ copies of $K_n$ is $$\begin{pmatrix}
		-2(n-1)^2 & nk-2k & A & B\\
		nk-2k & k-1 & 1 & 1 \\
	\end{pmatrix},$$ where $A=2(n-1)^2(n-2)+\dfrac{ \sqrt{2(n-1)^2(2-n)-4k(n-1)^2(1-n)}}{2}$ and
	
	$B=2(n-1)^2(n-2)-\dfrac{ \sqrt{2(n-1)^2(2-n)-4k(n-1)^2(1-n)}}{2}.$
\end{proof}
\section{Lanzhou inertia of path}
\begin{theorem}
	The Lanzhou	inertia of path $P_{2n}$ is $In(A_{Lz}(P_{2n}))=(n,n,0).$
\end{theorem}
\begin{proof}
	The Lanzhou matrix of $P_{2n}$ is given by
	$$A_{Lz}(P_{2n})=\footnotesize\begin{+bmatrix} [
		colspec={@{~}ccccccccc@{~}},
		colsep=0.1pt
		]
		0 & 2(3n-4) & 0 & 0&  \cdots & 0 & 0 & 0\\
		2(3n-4) & 0 & 4(2n-3) & 0 & \cdots & 0 & 0 & 0\\
		0 & 4(2n-3) & 0 & 4(2n-3) & \cdots & 0 &0 & 0\\
		\vdots& \vdots& \vdots& \vdots& \ddots & \vdots &\vdots & \vdots\\
		0 & 0& 0& 0 & \cdots & 4(2n-3) & 0& 2(3n-4)\\
		0 & 0& 0& 0&  \cdots & 0 & 2(3n-4) & 0\\
	\end{+bmatrix}_{2n}.$$
	
	
	Now consider the principal submatrix, $B$ of order $2n-1.$
	$$B=\footnotesize\begin{+bmatrix} [
		colspec={@{~}ccccccccc@{~}},
		colsep=0.1pt
		]
		0 & 2(3n-4) & 0 & 0&  \cdots & 0 & 0 \\
		2(3n-4) & 0 & 4(2n-3) & 0 & \cdots & 0 & 0 \\
		0 & 4(2n-3) & 0 & 4(2n-3) & \cdots & 0 &0 \\
		\vdots& \vdots& \vdots& \vdots& \ddots & \vdots &\vdots \\
		0& 0& 0& 0& \cdots& 0&4(2n-3)\\
		0 & 0& 0& 0 & \cdots & 4(2n-3) & 0\\
	\end{+bmatrix}_{2n-1}.$$
	
	Let $H_{11}=\begin{bmatrix}
		0 & 2(3n-4)\\
		2(3n-4) & 0\\
	\end{bmatrix}_{2},$  $H_{12}=\begin{bmatrix}
		0 & 0&  \cdots & 0 & 0 \\
		4(2n-3) & 0 & \cdots & 0 & 0 \\
	\end{bmatrix}_{2\times 2n-3},$ 
	
	$H_{21}=\begin{bmatrix}
		0 & 4(2n-3)\\
		0 & 0\\
		\vdots & \vdots\\
		0 & 0\\
	\end{bmatrix}_{2n-3 \times 2},$ and $$H_{22}=\footnotesize\begin{+bmatrix} [
		colspec={@{~}ccccccccc@{~}},
		colsep=0.2pt
		]
		0 & 4(2n-3) & 0 &   \cdots & 0 & 0 \\
		4(2n-3) & 0 & 4(2n-3) &  \cdots & 0 & 0 \\
		0 & 4(2n-3) & 0 &  \cdots & 0 &0 \\
		\vdots& \vdots& \vdots&  \ddots & \vdots &\vdots \\
		0 & 0& 0&  \cdots & 0 & 4(2n-3)\\
		0 & 0& 0&  \cdots & 4(2n-3) & 0\\
	\end{+bmatrix}_{2n-3}.$$
	\normalsize
	Since $H_{11}$ is non-singular, by Lemma [\ref{4.8}], $$In(B)=In(H_{11})+In(B/H_{11}),$$ where $B/H_{11}=H_{22}-H_{21}{H_{11}}^{-1}H_{12}.$

	Now clearly, $In(H_{11})=(1,1,0)$ and $H_{22}-H_{21}{H_{11}}^{-1}H_{12}=H_{22}.$

	So, $In(B)=In(H_{11})+In(H_{22})=(n-1,n-1, 1).$

	Let $\lambda_1\geq \lambda_2\geq \cdots\geq \lambda_{2n}$ and $\mu_1\geq \mu_2 \geq \cdots\geq \mu_{2n-1}$  be the eigenvalues of $A_{Lz}(P_{2n})$ and $B,$ respectively.

	Now, we have in total $n-1$ positive, $n-1$ negative and one zero eigenvalues in $B$ i.e., $$\mu_1\geq \mu_2\geq \cdots\geq \mu_{n-1}> 0> \mu_{n+1}\geq \cdots\geq \mu_{2n-1}.$$
	By the interlacing theorem, $$\lambda_{2n}\leq\mu_{2n-1}\leq \lambda_{2n-1}\leq \cdots< \mu_{n}< \lambda_{n}< \mu_{n-1}\leq \lambda_{n-1}\leq \cdots\leq \mu_1\leq \lambda_1.$$

	Clearly $\mu_{n}=0$ and $\lambda_n>\mu_n>\lambda_{n+1}.$
	
	$\implies \lambda_n>0 \text{\ and \ } \lambda_{n+1}<0.$

	So, $A_{Lz}(P_{2n})$ has $n$ positive and $n$ negative eigenvalues.

	Therefore, $In(A_{Lz}(P_{2n}))=(n, n, 0).$
\end{proof}
\begin{theorem}
The	Lanzhou	inertia of path $P_{2n+1}$ is given by
	$$In(A_{Lz}(P_{2n+1}))=(n,n,1).$$
\end{theorem}
\begin{proof}
	We have
	$$A_{Lz}(P_{2n+1})=\footnotesize\begin{+bmatrix} [
		colspec={@{~}ccccccccc@{~}},
		colsep=0.2pt
		]
		0 & 6n-5 & 0 & 0&  \cdots & 0 & 0 & 0\\
		6n-5 & 0 & 8(n-1) & 0 & \cdots & 0 & 0 & 0\\
		0 & 8(n-1) & 0 & 8(n-1) & \cdots & 0 &0 & 0\\
		\vdots& \vdots& \vdots& \vdots& \ddots & \vdots &\vdots & \vdots\\
		0 & 0& 0& 0 & \cdots & 8(n-1) & 0& 6n-5\\
		0 & 0& 0& 0&  \cdots & 0 & 6n-5 & 0\\
	\end{+bmatrix}_{2n+1}.$$

	Now, consider the principal sub matrix, $B$
	$$B=\footnotesize\begin{+bmatrix} [
		colspec={@{~}ccccccccc@{~}},
		colsep=0.2pt
		]
		0 & 6n-5 & 0 & 0&  \cdots & 0 & 0 \\
		6n-5 & 0 & 8(n-1) & 0 & \cdots & 0 & 0 \\
		0 & 8(n-1) & 0 & 8(n-1) & \cdots & 0 &0 \\
		\vdots& \vdots& \vdots& \vdots& \ddots & \vdots &\vdots \\
		0& 0& 0& 0& \cdots& 0&8(n-1)\\
		0 & 0& 0& 0 & \cdots & 8(n-1) & 0\\
	\end{+bmatrix}_{2n}.$$
	From the above theorem, $In(B)=(n,n,0).$

	Let $\lambda_1\geq \lambda_2\geq \cdots\geq \lambda_{2n+1}$ and $\mu_1\geq \mu_2 \geq \cdots\geq \mu_{2n}$ be the eigenvalues of $A_{Lz}(P_{2n+1})$ and $B,$ respectively.

	Now, we have in total $n$ positive, $n$ negative eigenvalues in $B,$ that is, $$\mu_1\geq \mu_2\geq \cdots\geq \mu_{n}> \mu_{n+1}\geq \cdots\geq \mu_{2n}.$$
	
	By the interlacing theorem, $$\lambda_{1}\geq\mu_{1}\geq \lambda_{2}\geq \cdots\geq  \lambda_{n}\geq \mu_n> \lambda_{n+1}> \mu_{n+1}\geq \cdots\geq \lambda_{2n}\geq \mu_{2n}\geq \lambda_{2n+1}.$$

	Clearly $\lambda_{n+1}=0.$
	So, $A_{Lz}(P_{2n+1})$ has $n$ positive, $n$ negative and one zero eigenvalues.

	Therefore, $In(A_{Lz}(P_{2n+1}))=(n, n, 1).$
\end{proof}

\section{Characterizations for symmetricity of the Lanzhou eigenvalues about the origin}

\begin{lemma}\label{ddd}
Let $\Gamma$ be a simple, undirected graph. Let $A_w(\Gamma)$ be the real, non-negative weighted adjacency matrix associated with the graph $\Gamma.$ Then the eigenvalues of $A_w(\Gamma)$ are symmetric about the origin (if $\lambda$ is an eigenvalue of $A_w(\Gamma)$ with multiplicity $k$ then $-\lambda$ is also an eigenvalue of $A_w(\Gamma)$ with multiplicity $k$) if and only if $A_w(\Gamma)$ is of the form  $$\begin{bmatrix}
0 & B\\
B^T & 0\\
\end{bmatrix}.$$
\end{lemma}
\begin{proof}
Let $A_w(G)$ be the real, non-negative weighted adjacency matrix associated with the graph $G.$ 
	
Let $A_w(G)=\begin{bmatrix}
0 & B\\
B^T & 0\\
\end{bmatrix}.$ To prove eigenvalues of $A_w(G)$ are symmetric about the origin,	let $\mathbf{x}=[\mathbf{y}\ \mathbf{z}]^{T}$ be an eigenvector of $A_{w}(G)$ with $\gamma$ as the eigenvalue. Then $$\begin{bmatrix}
		0 & B\\
		B^{T} & 0\\
	\end{bmatrix}\begin{bmatrix}
		\mathbf{y}\\\mathbf{z}
	\end{bmatrix}=\gamma \begin{bmatrix}
		\mathbf{y}\\\mathbf{z}
	\end{bmatrix}.$$
	This implies, $B\mathbf{z}=\gamma \mathbf{y}$ and $B^T \mathbf{y}=\gamma \mathbf{z}.$
	
	Define $\mathbf{x'}=\begin{bmatrix}
		\mathbf{y} \\\mathbf{-z}
	\end{bmatrix}\neq 0.$
	
	$$\begin{bmatrix}
		0 & B\\
		B^{T} & 0\\
	\end{bmatrix}\begin{bmatrix}
		\mathbf{y}\\\mathbf{-z}
	\end{bmatrix}=\begin{bmatrix}
		-B\mathbf{z}\\B^{T}\mathbf{y}
	\end{bmatrix}=\begin{bmatrix}
		-\gamma \mathbf{y}\\\gamma
        \mathbf{z}
	\end{bmatrix}=-\gamma \begin{bmatrix}
		\mathbf{y}\\\mathbf{-z}
	\end{bmatrix}.$$
	
	This implies $-\gamma$ is also an eigenvalue with eigenvector $[\mathbf{y}\ \mathbf{-z}]^{T}.$

    Conversely, let the eigenvalues of $A_w(G)$ be symmetric about the origin.
	
	Let \([A_{w}(G )^{k}]_{ij}\) denote the \((i,j)\)-th entry of the matrix \(A_{w}(G )\) raised to the \(k\)-th power.
	This implies $-\gamma $ is also an eigenvalue with eigenvector $[y\ -z]^{T}.$

	We have
	$$[A_{w}(G )^{k}]_{ij}=\sum\limits_{\substack{closed \ walks \\ of \ length\  k\\  starting\\  at\  i\  and \\ ending\ at\  j}}\prod\limits_{\substack{edges\ in \\ the \ walk\\ i\leq u<v\leq j}}[A_{w}(G )]_{uv}.$$
Since the eigenvalues of $A_w(G)$ are symmetric about the origin, we have
$trace[A_{w}(G)^k]=0,$ for all odd $k.$ 
Since $A_w(G)$ is non-negative, there is no cycle of odd length in $G.$ 

$\implies$ $G$ is bipartite. With a proper labeling of $V(G)$, we can find a a non-trivial partition of $V(G)$ such that $$A_w(G)=\begin{bmatrix}
0 & B\\
B^T & 0\\
\end{bmatrix}.$$ 
\end{proof}
\begin{theorem}\label{4.19}
The eigenvalues of $A_{Lz}(\Gamma)$ are symmetric about the origin if and only if there exists a non-trivial partition $X$ and $Y$ of $V(\Gamma)$ such that one of the following is true. \begin{enumerate}
\item  Both $V(X)$  and $V(Y)$ forms a clique with all vertices of full degree. 
\item $V(X)$ forms a clique with all vertices of full degree and $V(Y)$ induce an independent set.
\item $V(Y)$ forms a clique with all vertices of full degree and $V(X)$ induce an independent set.
\item  Both $V(X)$  and $V(Y)$ induce an independent set. Here, if no edges are present between $V(X)$ and $V(Y),$ then $\Gamma=\overline{K_n}.$ If there are some edges between $X$ and $Y,$ then $\Gamma$ is bipartite.
\end{enumerate}
\end{theorem}
\begin{proof}
	Let the eigenvalues of $A_{Lz}(\Gamma)$ be symmetric about the origin. Since $A_{Lz}(\Gamma)$ is non-negative, by Lemma [\ref{ddd}], there exists a non-trivial bi-partition of $V(\Gamma)$  such that $A_{Lz}(\Gamma)=\begin{bmatrix}
		0 & B\\
		B^T & 0\\
	\end{bmatrix}.$ 
	Let \begin{equation} A_{Lz}(\Gamma)= \label{matrix}
		\begin{bmatrix}
			A & B\\
			B^T & D\\
\end{bmatrix}.
\end{equation}
According to the definition of Lanzhou matrix of a graph, the entries of $A_{Lz}(\Gamma)=0$ if and only if $\Gamma$ is completely disconnected graph, or $\Gamma$ forms a clique with degree of the vertex as $n-1$ for each vertex.
Therefore, both $A$ and $D$ are zero if and only if there is a bi-partition $X,\ Y$ of $V(\Gamma)$ corresponding to the vertices in $A$ and $D,$ respectively, such that either one of the following holds.
\begin{enumerate}
\item both $V(X)$ and $V(Y)$ forms a clique with all vertices of full degree. 
\item $V(X)$ forms a clique with all vertices of full degree and $V(Y)$ induce an independent set,
\item $V(Y)$ forms a clique with all vertices of full degree and $V(X)$ induce an independent set.
\item  Both $V(X)$ and $V(Y)$ induce an independent set. Here, if no edges are present between $X$ and $Y,$ then $\Gamma=\overline{K_n}.$ If there are some edges between $X$ and $Y,$ then $\Gamma$ is bipartite.
\end{enumerate}
Converse holds trivially.
\end{proof}

\section{Conclusion}

Graph energy, a concept rooted in spectral graph theory, has significant applications in chemistry, particularly in the study of molecular structures and chemical stability.

In this paper, we define the Lanzhou matrix and its corresponding Lanzhou energy. We establish upper and lower bounds for 
$E_{Lz}(\Gamma)$ and $Sp(A_{Lz}(\Gamma)).$ Additionally, we compute the Lanzhou spectrum for cycle graphs, crown graphs, barbell graphs, cocktail party graphs, and the coalescence of $K_n.$ We also determine the Lanzhou inertia of path graphs and provide a characterization for graphs with symmetric Lanzhou eigenvalues.

\end{document}